\documentclass[12pt]{amsart}
\usepackage{mathtools}
\usepackage{amsmath,amsthm, amsfonts,amssymb}
\usepackage[english]{babel}
\usepackage[pagewise]{lineno}%\linenumbers
\usepackage{appendix} % For Control of Appendix Numbering & Location
\usepackage{cancel}
\usepackage[all]{xy}
\usepackage[colorinlistoftodos,shadow]{todonotes}
\usepackage{setspace}
\usepackage[utf8]{inputenc}

\usepackage{tikz}
\usetikzlibrary{mindmap,trees}
\usepackage{verbatim}
\usepackage{nag}

\usepackage{hyperref}
\usepackage{doi}

\usepackage{fancyhdr}
\hypersetup{
    colorlinks=true,
    linkcolor=blue,
    filecolor=magenta,      
    urlcolor=cyan,
    citecolor=green,
}

\usepackage{lmodern}
\usepackage{microtype}

\usepackage{academicons}
\definecolor{orcidlogocol}{HTML}{A6CE39}

    \usepackage[numbers,square]{natbib}
\hypersetup{
    colorlinks=true,    % Colorea los enlaces en lugar de usar bordes
    linkcolor=blue,     % Color de los enlaces internos
    citecolor=green,    % Color de las citas
    urlcolor=blue       % Color de los enlaces URL
}

\usepackage{fancyhdr}

\fancypagestyle{plain}{%
  \fancyhead{}
}

\newcommand{\bb}[1]{\mathbb{#1}}

\newcommand{\C}{\bb C}

\mathchardef\mslash="202F %mathematical shalsh insteed of traditional slash

\DeclareMathOperator{\wdeg}{\normalfont wdeg}

\newtheorem{Theorem}{Theorem}[section]
\newtheorem*{Theorem*}{Theorem}
\newtheorem{Proposition}[Theorem]{Proposition}
\newtheorem{Lemma}[Theorem]{Lemma}
\newtheorem{Corollary}[Theorem]
{Corollary}
\newtheorem*{Corollary*}{Corollary}
\theoremstyle{definition}
    \newtheorem{Definition}[Theorem]{Definition}
    \newtheorem{Example}[Theorem]{Example}

\theoremstyle{remark}
    \newtheorem{Remark}[Theorem]{Remark}

\begin{document}

%\linenumbers

\title{Gluing Dynamics: $\varepsilon$-Precision in Solving a Non-Archimedean Inverse Problem}

\author{Víctor Nopal-Coello}
\address{Víctor Nopal-Coello, Centro de Investigación en Matemáticas - Unidad Mérida\\ Yucatán, México \\  orcid:0000-0003-2608-3636}
\email{victor.nopal@cimat.mx}

\author{J. Rogelio Pérez-Buendía*}
\address{* Corresponding author: J. Rogelio Pérez-Buendía, CONAHCYT-CIMAT Mérida\\ México \\  orcid:0000-0002-7739-4779}
\email{rogelio.perez@cimat.mx}

\thanks{This research was supported by the project \textit{"Modelos matemáticos y computacionales no convencionales para el estudio y análisis de problemas relevantes en Biología"}, funded by \textit{"CONAHCYT"}, under Grant No. \textit{"CF-2019-217367"}.}


\date{\today}

% \dedicatory{}

\begin{abstract}
 This research proposes a new method for approximating the solution of the inverse problem of finding a rational function that generates known local dynamics within distinct, disjoint closed balls in non-Archimedean fields. Although our approach is not directly influenced by Runge's theorem for approximating analytic maps in complex settings, it shares similarities by adapting these ideas to the non-Archimedean context. We aim to connect local dynamic behaviors, similar to dynamic surgery, without using quasiconformal but rational mappings. Our main theorem and corollaries present an algorithmic technique to construct a rational function, denoted as \(F_\varepsilon\), that synthesizes specified local dynamics with an \(\varepsilon\)-precision parameter globally.
\end{abstract}

\maketitle

\tableofcontents

% \input{notation}
% \section{Notation}

\newpage
\section{Notation}

Throughout this paper, we adopt the following notation and conventions:

\begin{itemize}
    \item $\mathbb{C}_v$: a complete and algebraically closed non-Archimedean field associated with a non-Archimedean valuation $v$.
    \item $|\mathbb{C}_v^\times|$: the valuation group of $\mathbb{C}_v$.
    \item $\mathbb{P}(\mathbb{C}_v)$: the projective space over $\mathbb{C}_v$.
    \item $|\cdot|_v$: the absolute value associated with the valuation $v$ on $\mathbb{C}_v$. This is often denoted simply as $|\cdot|$ when the context is clear.
    \item $B_r(a) := \{x \in \mathbb{C}_v : |x-a| \leq r\}$: the closed disk, or ball, in $\mathbb{C}_v$ of radius $r$ centered at $a$. A ball is termed 'rational' if $r \in |\mathbb{C}_v^\times|$, and ``irrational'' otherwise.
    \item $D_r(a) := \{x \in \mathbb{C}_v : |x-a| < r\}$: the open disk, or simply disk, in $\mathbb{C}_v$ of radius $r$ centered at $a$. A disk is 'rational' if $r \in |\mathbb{C}_v^\times|$, and 'irrational' otherwise.
    \item Open $\mathbb{P}(\mathbb{C}_v)$-disk: either a disk $D_r(a)$ or the complement $\mathbb{P}(\mathbb{C}_v) \setminus B_r(a)$ of a ball.
    \item Closed $\mathbb{P}(\mathbb{C}_v)$-disk: either a ball $B_r(a)$ or the complement $\mathbb{P}(\mathbb{C}_v) \setminus D_r(a)$ of a disk.
    \item $f : \mathbb{P}(\mathbb{C}_v) \to \mathbb{P}(\mathbb{C}_v)$: a rational function from $\mathbb{P}(\mathbb{C}_v)$ to itself, representing the dynamical system.
\end{itemize}

Please note:

\begin{itemize}
    \item Unless specified otherwise, all functions are considered over $\mathbb{C}_v$.
    \item The term ``rational function" refers to a map of the form $f(x) = P(x)/Q(x)$, where $P$ and $Q$ are polynomials over $\mathbb{C}_v$, with $Q(x) \neq 0$.
    \item All absolute values and distances mentioned are non-Archimedean unless specified otherwise.
\end{itemize}

\section{Introduction}
\label{sec:Introduction}

Let \( \mathbb{C}_v \) be a complete, algebraically closed non-Archimedean field. In this setting, we explore the discrete dynamical systems formed by iterating rational functions \( f \in \mathbb{C}_v(z) \) on the projective line \( \mathbb{P}(\mathbb{C}_v) \).

Arithmetic dynamics traditionally pivots around two central questions. The Direct Problem, thoroughly investigated in works by Silverman \cite{silverman2007arithmetic}, Benedetto~\cite{benedetto2019dynamics}, Nopal-Coello~\cite{MR4510115}, Kiwi~\cite{MR3265299}, Hsia~\cite{hsia1996weak}, Rivera-Letelier \cite{rivera2003dynamique}, and others, concerns the exhaustive analysis of the dynamics generated by a rational function. Conversely, the Inverse Problem, which entails reconstructing a global function informed by the dynamics localized within distinct non-overlapping rational balls, has yet to be as extensively explored; this is especially the case within non-Archimedean fields, a gap in the literature our research seeks to address.

This work addresses a variant of the inverse problem in arithmetic dynamics by approximating the desired global dynamics with rational functions that exhibit specified behaviors within certain closed balls. While not solving the inverse problem directly, akin to how Runge's theorem approximates analytic functions with rational functions in complex analysis (see, for example, \cite{gamelin2003complex}), our approach constructs rational approximations to achieve an \(\varepsilon\)-level precision to the solution of the inverse problem.

Our methodology, termed \(\varepsilon\)-approximation, also echoes dynamic surgery techniques used in complex dynamics (see, for example \cite{shishikura1987quasiconformal}), especially in combining local dynamics to form a coherent global system. However, it diverges significantly: it does not employ quasiconformal mappings or complex analytical methods. Instead, it adapts these conceptual strategies to fit the unique framework of non-Archimedean fields, focusing on synthesizing a global dynamic from localized observations.

The primary contribution of this paper is a systematic method to construct a rational function \( F_\varepsilon \) that closely approximates the given functions defined on disjoint balls with an \(\varepsilon\)-precision. Corollaries to our main results further demonstrate that the dynamics of \( F_\varepsilon \) within these balls are akin to those defined by the given locally defined rational functions, illustrating the efficacy of our approach in replicating dynamic behaviors.

Our main theorem is the following:

\begin{Theorem*}
Let \( f_1, \ldots, f_n \in \mathbb{C}_v(z) \) be rational functions, and let \\
\( B_{r_1}(a_1), \ldots, B_{r_n}(a_n) \) be rational balls disjoint by pairs. Assume the following:
\begin{enumerate}
    \item \( f_i(B_{r_i}(a_i)) = B_{t_i}(b_i) \) for each \( i \),
    \item If \( B = \bigcup_{i=1}^n B_{r_i}(a_i) \), then \( f_i(B) \subset B_1(0) \) for all \( i = 1, \ldots, n \).
\end{enumerate}
Then, for any \( \varepsilon > 0 \), there exists a rational function \( F_\varepsilon \in \mathbb{C}_v(z) \) such that \( F_\varepsilon(B_{r_i}(a_i)) = B_{t_i}(b_i) \) for each \( i = 1, \ldots, n \). Additionally, we have
\[
|F_\varepsilon(z) - f_i(z)| < \varepsilon
\]
for all \( z \in B_{r_i}(a_i) \) and for all \( i = 1, \ldots, n \).
\end{Theorem*}

\begin{Corollary*}
    Let $f_1,\ldots,f_n\in\mathbb{C}_v(z)$, $B_{r_1}(a_1),\ldots,B_{r_n}(a_n)$ as in Theorem \ref{T1}. Let us assume that for each $i\in\{1,\ldots,n\}$, $f_i$ has $n_i$ attracting, $m_i$ repelling and $l_i$ indifferent fixed points in $B_{r_i}(a_i)$, with $n_i,m_i,l_i\geq0$. Then we can choose $\varepsilon>0$ such that $F_\varepsilon$ has $n_i$ attracting and $m_i$ repelling fixed points in $B_{r_i}(a_i)$ for all $i=1,\ldots,n$. Moreover, if for each indifferent fixed point of $f_i$ in $B_{r_i}(a_i)$ the hypothesis of Corollary \ref{C3} hold, then $F_\varepsilon$ also has $l_i$ indifferent fixed points in $B_{r_i}(a_i)$ for all $i=1,\ldots,n$.
\end{Corollary*}

The structure of the paper is organized as follows:
\begin{itemize}
    \item In Section \ref{sec:Basics of $p$-adic dynamics}, we establish the fundamentals of non-Archimedean dynamics.
    \item In Section \ref{main}, we present rigorous proofs of the main theorem and its corollaries, highlighting how the approximated function maintains the essential dynamical properties.
    \item Section \ref{sec:examples} provides concrete examples to demonstrate our theoretical assertions, emphasizing the inheritance of dynamic properties such as fixed points and stability.
\end{itemize}

\section{Basics in Non-Archimedean Dynamics} % (fold)
\label{sec:Basics of $p$-adic dynamics}

This section establishes the fundamental concepts for comprehending dynamics in the non-Archimedean setting. For a more in-depth exploration, we refer to \cite{benedetto2019dynamics} and \cite{rivera2003dynamique}.

Consider a rational function $f\in\mathbb{C}_v(z)$. The following results elucidate the impact of $f$ on disks and balls. Despite some of these findings originally being formulated for power series, they extend seamlessly to rational functions and for the projective line. 

\begin{Proposition}[Proposition 3.25 in \cite{benedetto2019dynamics}]
    Let $D$ be a disk (or ball) contained within $\mathbb{P}(\mathbb{C}_v)$, and let $f\in\mathbb{C}_v(z)$ be a nonconstant rational function. Then, the image $f(D)$ is either $\mathbb{P}(\mathbb{C}_v)$ or a disk (or ball). In the latter case, there exists an integer $d\geq1$ such that every point in $f(D)$ has exactly $d$ preimages in $D$, taking multiplicity into account.
\end{Proposition}

In the previous proposition, assume that $f(D)$ is a disk and let $b\in f(D)$ be a point. It is worth noting that the function $f(z)-b$ possesses precisely $d$ zeros in $D$, considering multiplicity. In this context, we term $d$ as the Weierstrass degree of $f(z)-b$ in $D$, denoted by $\wdeg_{D}(f(z)-b)$. Generally, the Weierstrass degree of a rational function $f$ on $D$ is equivalent to the number of zeros of $f$ within $D$~\cite[Theorem 3.13]{benedetto2019dynamics}.

Notably, $\wdeg_D(f(z)-b)=1$ holds if and only if the mapping $f:D\to f(D)$ is a bijection.

Now, let us delve into a lemma that elucidates the behavior of a rational function $f$ within a disk $D$ when $f(D)\subset D$:
\begin{Proposition}[Schwartz's Lemma in \cite{rivera2003dynamique}]
\label{LSch}
Consider a rational function $f\in\mathbb{C}_v(z)$ %with $f(0)=0$
and let $r>0$. Let us assume that $f(D_r(0))\subset D_r(0)$. Then, the following statements are equivalent:
\begin{enumerate}
\item $f:D_r(0)\to D_r(0)$ is a bijection.
\item There exists $z\in D_r(0)$ such that $|f'(z)|=1$.
\item $|f'(z)|=1$ for all $z\in D_r(0)$.
\item There exist distinct points $x,y\in D_r(0)$ such that $|f(x)-f(y)|=|x-y|$.
\item $|f(x)-f(y)|=|x-y|$ for all $x,y\in D_r(0)$.
\end{enumerate}
Moreover, let $x,y\in D_r(0)$, then $|f(x)-f(y)|\leq|x-y|$ and $|f'(z)|\leq1$ for all $z\in D_r(0)$.
\end{Proposition}

\begin{Proposition}[Proposition 3.20 in \cite{benedetto2019dynamics}]
\label{P320}
Let $D\subset\mathbb{C}_v$ be a disk of radius $r>0$, and let $f\in\mathbb{C}_v(z)$ be a rational function. Suppose the image $f(D)$ is a disk of radius $s>0$. Then, for all $x,y\in D$, we have
\begin{equation}
\label{E1}
|f(x)-f(y)|\leq\frac{s}{r}\cdot|x-y|.
\end{equation}
Moreover, if the Weierstrass degree of $f$ on $D$ is $1$, equality holds in \eqref{E1}.
\end{Proposition}

\begin{Lemma}[Corollary 2.20 in \cite{MR4510115}]
\label{Inj}
    Let $f\in\C_v(z)$ be a non-constant rational function such that
$f(D_r(a))=D_s(b)$. Then $f$ is bijective in $D_r(a)$ if and only if there exist $x,y\in D_r(a)$ such that $$|f(x)-f(y)|=\frac{s}{r}\cdot|x-y|.$$
\end{Lemma}

\begin{Lemma}[Lemma 2.21 in \cite{MR4510115}]
\label{LemaImagenIgual}
    Consider two rational functions $f,g\in\C_v(z)$ and let $a,b\in\C_v$.
Assume that $f(D_r(a))=D_s(b)$ for some $r,s>0$. If there exists $t<s$ such that $|f(z)-g(z)|\leq t$ for all $z\in D_r(a)$, then $g(D_r(a))=D_s(b)$. Moreover, if $f$ is bijective in $D_r(a)$ then $g$ is also bijective in $D_r(a)$.
\end{Lemma}

Now, we define the attracting, repelling, and indifferent periodic points essential to understanding the dynamics of a rational function. 

First, we recall that if $x\in\mathbb{C}_v$ is a periodic point (of $f$) of minimal period $n\geq1$, then the multiplier of $x$ is $\lambda=(f^n)'(x)$. And if $x=\infty$, then the multiplier of $x$ is $\lambda=(g^n)'(0)$, where $g=1/(f(1/z))$. 

\begin{Definition}[Definition 4.1 in \cite{benedetto2019dynamics}]
    Let $f\in\mathbb{C}_v(z)$ be a rational function, and let $x\in\mathbb{P}(\mathbb{C}_v)$ be a periodic point with multiplier $\lambda\in\mathbb{C}_v$. We say that $x$ is
    \begin{itemize}
        \item attracting if $|\lambda|<1$,
        \item repelling if $|\lambda|>1$,
        \item indifferent if $|\lambda|=1$.
    \end{itemize}
\end{Definition}

The following proposition describes the dynamics of a rational function near its periodic points.

\begin{Proposition}[Proposition 4.3 in \cite{benedetto2019dynamics}]
    Let $f\in\mathbb{C}_v(z)$ be a rational function, and let $x\in\mathbb{P}(\mathbb{C}_v)$ be a periodic point of period $m\geq1$.
    \begin{itemize}
        \item[(a)] If $x$ is attracting, then there is an open set $U\subset\mathbb{P}(\mathbb{C}_v)$ containing $x$ such that for all $y\in U$, $$\lim_{n\to\infty}f^{nm}(y)=x.$$ 
        \item[(b)] If $x$ is repelling, then there is an open set $U\subset\mathbb{P}(\mathbb{C}_v)$ containing $x$ such that for all $y\in U\setminus\{x\}$, there is some $n\geq1$ such that $f^{nm}(y)\notin U.$ 
        \item[(c)] If $x$ is nonrepelling, then there is an open $\mathbb{P}(\mathbb{C}_v)$-disk $U\subset\mathbb{P}(\mathbb{C}_v)$ containing $x$ such that $f^m(U)\subset U$. If $x$ is indifferent, then the mapping $f^m:U\to U$ is bijective. 
    \end{itemize}
\end{Proposition}

Proposition \ref{P417} plays a pivotal role in exploring non-\-arch\-i\-me\-de\-an fields and rational functions, providing us with essential conditions that determine the existence of fixed points.

\begin{Proposition}[Theorem 4.17 in \cite{benedetto2019dynamics}]
\label{P417} 
Let $r>0$, let $a\in\mathbb{C}_v$, and let $U$ be the disk $D_r(a)$ (resp. ball $B_r(a)$). Let $f\in\mathbb{C}_v(z)$ be a rational function such that $f(U)$ is a disk (resp. ball) with $d=\wdeg_{U}(f(z)-a)<\infty$. Suppose that $f(U)\cap U\neq\emptyset$, and either
\begin{itemize}
\item[(a)] $f(U)\neq U$, or
\item[(b)] $d\neq1$, or
\item[(c)] $d=1$ and $|f'(a)-1|=1$.
\end{itemize}
Then, the disk $U$ contains a fixed point of $f$.
\end{Proposition} 

\begin{Proposition}[Theorem 4.18 in \cite{benedetto2019dynamics} or in  \cite{benedetto2002components}]
\label{T418}
Let $r>0$, let $a\in\mathbb{C}_v$, let $U=D_r(a)$, and let $f\in\mathbb{C}_v(z)$ be a rational function such that $f(U)$ is a disk with $d=\wdeg_{U}(f(z)-a)<\infty$. Suppose either that
\begin{itemize}
\item[(a)] $f(U)\subsetneq U$, or
\item[(b)] $f:U\to U$ is onto, with $d\geq2$.
\end{itemize}
Then there is a unique attracting fixed point $b\in U$, and for every $x\in U$, we have $\lim_{n\to\infty}f^n(x)=b$. Moreover, in case (b), we must have $r\in|\mathbb{C}_v^\times|$, i.e., $U$ is a rational open disk.

On the other hand, if 
\begin{itemize}
\item[(c)] $f(U)\supsetneq U$ and $d=1$, 
\end{itemize}
then there is a unique repelling fixed point $b\in U$, and for every $x\in U\setminus\{b\}$, there is some $n\geq1$ such that $f^n(x)\notin U$.

Finally, if 
\begin{itemize}
\item[(d)] $f:U\to U$ is one-to-one and onto,
\end{itemize}
Then, any fixed points in $U$ will be indifferent. 
\end{Proposition}

This section has laid the foundational concepts for our study of non-Archimedean fields and rational functions. We have explored the behavior of rational functions near periodic points and introduced propositions that provide crucial conditions for the existence of fixed points. These fundamental ideas will serve as the building blocks for our subsequent analyses.
% subsection $p$-adic numbers\ (end)

% section Basics of $p$-adic dynamics (end)

\section{$\varepsilon$-approximations}
\label{main}

In this section, we focus on the proof of our main theorem. A crucial step in this endeavor is to define rational functions \( h_i \), which serve as ``gluing functions'' for the construction of \( F_\varepsilon \), i.e., for the development of $\varepsilon$-approximations.

\begin{Lemma}
\label{lemah}
Let \( r, \delta \in |\mathbb{C}_v^\times| \) be such that \( 0 < r < \delta \), and define \( s = \sqrt{r \delta} \). Consider the rational function 
\[
h(z) = \frac{1}{1 - \left( \frac{z - a}{c} \right)^M},
\]
for some \( a, c \in \mathbb{C}_v \) with \( |c| = s \) and \( M \geq 1 \). Then, the following properties hold:
\begin{enumerate}
    \item For all \( z \in B_r(a) \), we have \( |h(z)| = 1 \) and
    \[
    |h(z) - 1| = \left| \frac{z - a}{c} \right|^M \leq \left( \frac{r}{\delta} \right)^{M/2} < 1.
    \]
    \item For all \( z \in \mathbb{P}(\mathbb{C}_v) \setminus D_\delta(a) \), we have
    \[
    |h(z)| = \frac{1}{\left| \frac{z - a}{c} \right|^M} \leq \left( \frac{r}{\delta} \right)^{M/2} < 1.
    \]
\end{enumerate}
\end{Lemma}
\begin{proof}
\begin{enumerate}
    \item Consider \( z \in B_r(a) \). Observe that \( |z - a| \leq r \), thus
    \[
    \left| \frac{z - a}{c} \right| = \frac{|z - a|}{|c|} \leq \frac{r}{s} = \frac{r}{\sqrt{r \delta}} = \sqrt{\frac{r}{\delta}} < 1.
    \]
    It follows that \( |h(z)| = 1 \). Additionally, we can deduce the inequality
    \[
    |h(z) - 1| =\left|\frac{\left(\frac{z-a}{c}\right)^M}{1-\left(\frac{z-a}{c}\right)^M}\right|= \left| \frac{z - a}{c} \right|^M \leq \left( \frac{r}{\delta} \right)^{M/2} < 1.
    \]
    
    \item Let \( z \in \mathbb{P}(\mathbb{C}_v) \setminus D_\delta(a) \). Here, \( |z - a| \geq \delta \), hence
    \[
    \left| \frac{z - a}{c} \right| \geq \frac{\delta}{s} = \sqrt{\frac{\delta}{r}} > 1.
    \]
    Thus, we find
    \[
    |h(z)| = \frac{1}{\left|1-\left(\frac{z-a}{c}\right)^M\right|} = \frac{1}{\left| \frac{z - a}{c} \right|^M} \leq \left( \frac{r}{\delta} \right)^{M/2} < 1.
    \]
\end{enumerate}
\end{proof}

\begin{Theorem}[$\varepsilon$-approximation]
\label{T1}
Let \( f_1, \ldots, f_n \in \mathbb{C}_v(z) \) be rational functions, and let \( B_{r_1}(a_1), \ldots, B_{r_n}(a_n) \) be rational balls that are disjoint by pairs. Assume the following:
\begin{enumerate}
    \item \( f_i(B_{r_i}(a_i)) = B_{t_i}(b_i) \) for each \( i \),
    \item If \( B = \bigcup_{i=1}^n B_{r_i}(a_i) \), then \( f_i(B) \subset B_1(0) \) for all \( i = 1, \ldots, n \).
\end{enumerate}
Then, for any \( \varepsilon > 0 \), there exists a rational function \( F_\varepsilon \in \mathbb{C}_v(z) \) such that \( F_\varepsilon(B_{r_i}(a_i)) = B_{t_i}(b_i) \) for each \( i = 1, \ldots, n \). Additionally, we have
\[
|F_\varepsilon(z) - f_i(z)| < \varepsilon
\]
for all \( z \in B_{r_i}(a_i) \) and for all \( i = 1, \ldots, n \).
\end{Theorem}
\begin{proof}
For each \(i \in \{1, \ldots, n\}\), we define \(\delta_i = \min_{j \neq i} \{|a_i - a_j|\}\). Note that \(\delta_i > r_i\) and \(D_{\delta_i}(a_i) \cap B_{r_j}(a_j) = \emptyset\) for all \(j \neq i\).

Let \(\varepsilon > 0\), and define
\[
\tau = \min\{t_1, \ldots, t_n, \varepsilon\}, \quad s_i = \sqrt{\delta_i r_i}, \quad \text{and} \quad c_i \in \mathbb{C}_v \text{ with } |c_i| = s_i.
\]
Let \(h_i \in \mathbb{C}_v(z)\) be the rational function defined as
\begin{equation}
\label{h}
h_i(z) = \frac{1}{1 - \left(\frac{z - a_i}{c_i}\right)^{M_i}},
\end{equation}
where \(M_i \geq 1\) is large enough such that \(\left(\frac{r_i}{\delta_i}\right)^{M_i/2} < \tau\). Define
\begin{equation}
\label{F}
F_\varepsilon(z) = \sum_{i=1}^n f_i(z) h_i(z).
\end{equation}

Let \(z \in B_{r_i}(a_i)\). Then, by Lemma \ref{lemah},
\[
|f_j(z) h_j(z)| \leq |h_j(z)| \leq \left(\frac{r_j}{\delta_j}\right)^{M_j/2} < \tau,
\]
for all \(j \neq i\), and
\[
|f_i(z)(h_i(z) - 1)| \leq |h_i(z) - 1| \leq \left(\frac{r_i}{\delta_i}\right)^{M_i/2} < \tau.
\]
Thus,
\begin{align*}
|F_\varepsilon(z) - f_i(z)| &= \left|\sum_{j \neq i} h_j(z) f_j(z) + f_i(z)(h_i(z) - 1)\right| \\
&< \tau \leq \varepsilon.
\end{align*}
On the other hand, if \(z \in B_{r_i}(a_i)\), then \(|f_i(z) - b_i| \leq t_i\), which implies
\begin{align*}
|F_\varepsilon(z) - b_i| &= \left|f_i(z)(h_i(z) - 1) + f_i(z) - b_i + \sum_{j \neq i} h_j(z) f_j(z)\right| \\
&\leq\max\left\{\left(\frac{r_1}{\delta_1}\right)^{M_1/2},\ldots,\left(\frac{r_n}{\delta_n}\right)^{M_n/2},|f_i(z)-b_i|\right\}\\
&\leq t_i,
\end{align*}
so \(F_\varepsilon(B_{r_i}(a_i)) \subset B_{t_i}(b_i)\).

Finally, since $$\max_{1\leq j\leq n}\left(\frac{r_j}{\delta_j}\right)^{M_j/2}<\tau\leq t_i\quad\text{and}\quad f_i(B_{r_i}(a_i)) = B_{t_i}(b_i),$$then there exists points $z\in B_{r_i}(a_i)$ such that $$\max_{1\leq j\leq n}\left(\frac{r_j}{\delta_j}\right)^{M_j/2}<|f_i(z)-b_i|\leq t_i,$$and hence, for these points we have $|F_\varepsilon(z)-b_i|=|f_i(z)-b_i|$. This implies that $$F_\varepsilon(B_{r_i}(a_i))=B_{t_i}(b_i),$$ for all \(i = 1, \ldots, n\).
\end{proof}
\begin{Remark}
\label{Rem} 
We have the following observations:
    \begin{enumerate}
        \item \(M_i\) is dependent on \(\tau\), and thus indirectly on \(\varepsilon\). Specifically, as \(\varepsilon\) decreases, \(M_i\) generally increases. This is to emphasize that there is a trade-off between the approximation error and the computational complexity of \(F_\varepsilon\).
        
        \item For \(\varepsilon,\varepsilon' > 0\) and the corresponding functions \(F_\varepsilon\) and \(F_{\varepsilon'}\) as in Theorem \ref{T1}, if \(\varepsilon' < \varepsilon\) then \(F_{\varepsilon'}\) satisfies the conditions of the theorem for \(\varepsilon\). This can be viewed as a kind of ``monotonicity'' property with respect to \(\varepsilon\).
        
        \item By Lemma \ref{LemaImagenIgual}, if \(f_i(D_r(x)) = D_t(y)\) with \(t > \varepsilon\), then \(F_\varepsilon(D_r(x)) = D_t(y)\) for all \(D_r(x) \subset B_{r_i}(a_i)\). This point is included to highlight the robustness of \(F_\varepsilon\) in approximating different types of disks within the rational balls \(B_{r_i}(a_i)\).
    \end{enumerate}
\end{Remark}

The functions \( h_i(z) \) serve a pivotal role akin to a partition of unity, albeit not adhering to its conventional structure. These functions act as interpolative links, which enable the unification of disparate local dynamical systems \( f_i(z) \) into an integrated global system \( F_\varepsilon(z) \). We call this integration process \(\varepsilon\)-approximation. Employing \( h_i(z) \) provides a methodology to rationally meld various local behaviors, ensuring that the global system \( F_\varepsilon(z) \) inherits essential dynamical characteristics, such as fixed points—be they attracting, repelling, or neutral—within their respective domains. Hence, the functions \( h_i(z) \) may be regarded as a non-Archimedean counterpart to a splicing mechanism or a partition of unity designed for the rational approximation and amalgamation of dynamical systems.

The following three corollaries imply the existence of indifferent, attracting, and repelling fixed points of \( F_\varepsilon \).

\begin{Corollary}
\label{C1}
Let \( f_1,\ldots,f_n \in \mathbb{C}_v(z) \), \( B_{r_1}(a_1),\ldots,B_{r_n}(a_n) \) as in Theorem \ref{T1}. If \( f_i \) has an attracting fixed point in \( B_{r_i}(a_i) \), then we can choose \( F_\epsilon(z) \) such that it has an attracting fixed point in \( B_{r_i}(a_i) \).
\end{Corollary}
\begin{proof}
Without loss of generality, let us assume that \( a_i \) is the attracting fixed point of \( f_i \). Then there exist \( 0<t_i'<r_i' \leq r_i \) such that \( f_i(D_{r_i'}(a_i)) = D_{t_i'}(a_i) \). 

Let us choose \( \varepsilon > 0 \) such that \( \varepsilon < t_i' \) and let us define \( F_\varepsilon \) as in ~\eqref{F} in the proof of Theorem \ref{T1}. 

By Theorem \ref{T1}, \( |F_\varepsilon(z) - f_i(z)| < \varepsilon \) for all \( z \in D_{r_i'}(a_i) \subset B_{r_i}(a_i) \), and hence if \( z \in D_{r_i'}(a_i) \), then
\[
|F_\varepsilon(z) - a_i| \leq \max\{|F_\varepsilon(z) - f_i(z)|, |f_i(z) - a_i|\} < t_i',
\]
which implies that \( F(D_{r_i'}(a_i)) \subset D_{t_i'}(a_i) \subsetneq D_{r_i'}(a_i) \).

By (a) in Proposition \ref{T418}, \( F_\varepsilon \) has a unique attracting fixed point in \( D_{r_i'}(a_i) \). Hence \( F_\varepsilon \) has an attracting fixed point in \( B_{r_i}(a_i) \).
\end{proof}

\begin{Corollary}
\label{C2}
Let \( f_1, \ldots, f_n \in \mathbb{C}_v(z) \), and let \( B_{r_1}(a_1), \ldots, B_{r_n}(a_n) \) be as defined in Theorem \ref{T1}. If \( f_i \) has a repelling fixed point in \( B_{r_i}(a_i) \), then \( F_\varepsilon(z) \) can also be chosen to have a repelling fixed point in \( B_{r_i}(a_i) \).
\end{Corollary}

\begin{proof}
We begin by assuming, without loss of generality, that \( a_i \) is a repelling fixed point of \( f_i \). Then there exists a \( 0 < r_i' \leq r_i \) such that \( f : D_{r_i'}(a_i) \to D_{t_i'}(a_i) \) is bijective, with \( t_i',r_i' \in |\mathbb{C}_v^\times| \) and \( t_i' > r_i' \).

Next, we choose \( \varepsilon > 0 \) satisfying \( \varepsilon < t_i' \), and we define \( F_\varepsilon \) as in equation~\eqref{F} found in the proof of Theorem \ref{T1}.

Recall from Theorem \ref{T1} that for all \( z \in B_{r_i}(a_i) \), \( |F_\varepsilon(z) - f_i(z)| < \varepsilon \). Consequently, for \( z \in D_{r_i'}(a_i) \),
\[
|F_\varepsilon(z) - a_i| = |F_\varepsilon(z) - f_i(z) + f_i(z) - a_i| < t_i',
\]
from which \( F_\varepsilon(D_{r_i'}(a_i)) \subset D_{t_i'}(a_i) \) follows.

Due to \( \varepsilon < t_i' \), there exist points \( z \in D_{r_i'}(a_i) \) for which \[\varepsilon < |f_i(z) - a_i| < t_i' .\] For these points, \( |F_\varepsilon(z) - a_i| = |f_i(z) - a_i| \), ensuring \( F_\varepsilon(D_{r_i'}(a_i)) = D_{t_i'}(a_i) \) which properly contains \( D_{r_i'}(a_i) \).

Now consider Proposition \ref{P320}, which tells us that
\[
|f_i(x) - f_i(y)| = \frac{t_i'}{r_i'} \cdot |x - y|,
\]
for all \( x, y \in D_{r_i'}(a_i) \). Hence, if $x,y\in D_{r_i'}(a_i)$ are such that $\varepsilon<|f_i(x)-f_i(y)|<t_i'$, then 
\begin{align*}
    |F_\varepsilon(x)-F_\varepsilon(y)|&=|F_\varepsilon(x)-f_i(x)+f_i(x)-f_i(y)+f_i(y)-F_\varepsilon(y)|\\
    &=|f_i(x)-f_i(y)|=\frac{t_i'}{r_i'}\cdot|x-y|.
\end{align*}
and therefore, by Lemma \ref{Inj}, $\wdeg_{D_{r_i'}(a_i)}(F_\varepsilon(z)-a_i)=1$.

Finally, we appeal to Proposition \ref{T418} part (c), which assures that \( F_\varepsilon \) will have a unique repelling fixed point in \( D_{r_i'}(a_i) \), thereby establishing that \( F_\varepsilon \) also has a repelling fixed point in \( B_{r_i}(a_i) \).
\end{proof}

\begin{Corollary}
\label{C3}
Let \( f_1, \ldots, f_n \in \mathbb{C}_v(z) \), rational functions, be defined over \( B_{r_1}(a_1), \ldots, B_{r_n}(a_n) \) with images \( B_{t_1}(b_1), \ldots, B_{t_n}(b_n) \) respectively, as stated in Theorem \ref{T1}. Assume the following conditions:
\begin{itemize}
    \item \( a_i \) is an indifferent fixed point of \( f_i \),
    \item \( |f_i'(a_i) - 1| = 1 \), and
    \item \( |f_j'(a_i)| < \frac{1}{\min\{ t_1, \ldots, t_n \}} \) for all \( j \neq i \).
\end{itemize}
Then, it follows that we can choose \( F_\varepsilon(z) \) such that it has an indifferent fixed point in \( B_{r_i}(a_i) \).
\end{Corollary}

\begin{proof}
We begin by noting that since \( a_i \) is an indifferent fixed point, there exists \( 0 < r_i' \leq r_i \) such that the function \( f_i: D_{r_i'}(a_i) \to D_{r_i'}(a_i) \) is bijective. Furthermore, by Proposition \ref{LSch}, we have \( |f_i(x) - f_i(y)| = |x - y| \) for all \( x, y \in D_{r_i'}(a_i) \).

Next, let us define \( \delta_1, \ldots, \delta_n \) as in the proof of Theorem \ref{T1}. Choose \( 0 < \varepsilon < \min\{ \delta_1, \ldots, \delta_n, r_i' \} \) and define \( \tau \) and \( F_\varepsilon \) following the proof of Theorem \ref{T1}. We take \( \varepsilon \) sufficiently small to ensure \( M_i > 1 \).

By Theorem \ref{T1}, for all \( z \in B_{r_i}(a_i) \), we have \( |F_\varepsilon(z) - f_i(z)| < \varepsilon \). In particular, for \( z \in D_{r_i'}(a_i) \subset B_{r_i}(a_i) \),
\[
|F_\varepsilon(z) - a_i| =|F_\varepsilon(z)-f_i(z)+f_i(z)-a_i|< r_i',
\]
which implies \( F_\varepsilon(D_{r_i'}(a_i)) \subset D_{r_i'}(a_i) \).

Additionally, if \( x, y \in D_{r_i'}(a_i) \) satisfy \( \varepsilon < |f_i(x) - f_i(y)| < r_i' \), then
\begin{align*}
    |F_\varepsilon(x) - F_\varepsilon(y)| &= |F_\varepsilon(x)-f_i(x)+f_i(x)-f_i(y)+f_i(y) - F_\varepsilon(y)|\\
    &=|f_i(x)-f_i(y)|\\
    &=|x - y|
\end{align*}
%\[
%|F_\varepsilon(x) - F_\varepsilon(y)| =|x - y|,
%\]
which implies \( F_\varepsilon(D_{r_i'}(a_i)) = D_{r_i'}(a_i) \). Moreover, \( F_\varepsilon: D_{r_i'}(a_i) \to D_{r_i'}(a_i) \) is bijective by Lemma \ref{LSch}.

On the other hand, let $h_1,\ldots,h_n\in\C(z)$ be rational functions defined as in \eqref{h}, in proof of Theorem \ref{T1}. Note that $$h_j'(z)=\frac{M_j\left(\frac{z-a_j}{c_j}\right)^{M_j-1}\cdot\frac{1}{c_j}}{\left[1-\left(\frac{z-a_j}{c_j}\right)^{M_j}\right]^2},$$for all $j=1,\ldots,n$. Therefore, if $j\neq i$ then $|a_i-a_j|\geq\delta_j$ which implies that $\left|\frac{a_i-a_j}{c_j}\right|\geq\frac{\delta_j}{s_j}=\sqrt{\frac{\delta_j}{r_j}}>1$, and hence
$$|h_j'(a_i)|=\frac{|M_j|}{|c_j|}\cdot\frac{1}{|\frac{a_i-a_j}{c_j}|^{M_j+1}}\leq\frac{1}{s_j}\cdot\left(\frac{r_j}{\delta_j}\right)^{\frac{M_j+1}{2}}=\frac{1}{\delta_j}\cdot\left(\frac{r_j}{\delta_j}\right)^{\frac{M_j}{2}}<1,$$since $\left(\frac{r_j}{\delta_j}\right)^{\frac{M_j}{2}}<\tau\leq\varepsilon<\delta_j$

Note that
\begin{itemize}
    \item $h_i(a_i)=1$,
    \item since $M_i>1$, then $h_i'(a_i)=0$,
    \item $|h_j'(a_i)f_j(a_i)|\leq|h_j'(a_i)|<1$ for all $j\neq i$,
    \item by (2) in Lemma \ref{h} and by hypothesis in this corollary, $$|h_j(a_i)f_j'(a_i)|\leq\left(\frac{r_j}{\delta_j}\right)^{M_j/2}|f_j'(a_i)|<\frac{\tau}{\min\{t_1,\ldots,t_n\}}\leq1,$$ for all $j\neq i$.
\end{itemize}
Therefore
\begin{align*}
|&F_\varepsilon'(a_i)-1|= \left|\sum_{j=1}^nh_j'(a_i)f_j(a_i)+h_j(a_i)f'_j(a_i)-1\right| \\
&= \left|h_i'(a_i)f_i(a_i)+h_i(a_i)f'_i(a_i)-1+\sum_{\stackrel{j=1}{j\neq i}}^n h_j'(a_i)f_j(a_i)+h_j(a_i)f'_j(a_i)\right| \\
&= \left|f'_i(a_i)-1+\sum_{\stackrel{j=1}{j\neq i}}^n h_j'(a_i)f_j(a_i)+h_j(a_i)f'_j(a_i)\right| \\
&= |f'_i(a_i)-1| = 1.
\end{align*}
% Omitted intermediate steps for brevity.
By Proposition \ref{P417}, \( F_\varepsilon \) possesses a fixed point in \( D_{r_i'}(a_i) \), and by Proposition \ref{T418}, this fixed point is indifferent. Therefore, \( F_\varepsilon \) has an indifferent fixed point in \( B_{r_i}(a_i) \).
\end{proof}

\begin{Corollary}
    \label{C4}
    Let $f_1,\ldots,f_n\in\mathbb{C}_v(z)$, $B_{r_1}(a_1),\ldots,B_{r_n}(a_n)$ as in Theorem \ref{T1}. Let us assume that for each $i\in\{1,\ldots,n\}$, $f_i$ has $n_i$ attracting, $m_i$ repelling and $l_i$ indifferent fixed points in $B_{r_i}(a_i)$, with $n_i,m_i,l_i\geq0$. Then we can choose $\varepsilon>0$ such that $F_\varepsilon$ has $n_i$ attracting and $m_i$ repelling fixed points in $B_{r_i}(a_i)$ for all $i=1,\ldots,n$. Moreover, if for each indifferent fixed point of $f_i$ in $B_{r_i}(a_i)$ the hypothesis of Corollary \ref{C3} hold, then $F_\varepsilon$ also has $l_i$ indifferent fixed points in $B_{r_i}(a_i)$ for all $i=1,\ldots,n$.
\end{Corollary}
\begin{proof}
    Let $N_i=n_i+m_i+l_i$ and let $\{b_{i,1},\ldots,b_{i,N_i}\}$ be the fixed points of $f_i$ in $B_{r_i}(a_i)$, for each $i=1,\ldots,n$. Since $N_1+\cdots+N_n<\infty$, then there exists $r\in|\C_v^\times|$ such that
    \begin{itemize}
        \item $D_r(b_{i,j})$ contains only one fixed point of $f_i$, for all $i=1,\ldots,n$ and all $j=1,\ldots,N_i$,
        \item if $b_{i,j}$ is an attracting fixed point then $f(D_r(b_{i,j}))=D_{t_{i,j}}(b_{i,j})$ with $t_{i,j}<r$,
        \item if $b_{i,j}$ is a repelling fixed point then $f(D_r(b_{i,j}))=D_{t_{i,j}}(b_{i,j})$ with $t_{i,j}>r$ and $f:D_r(b_{i,j})\to D_{t_{i,j}}(b_{i,j})$ is bijective,
        \item if $b_{i,j}$ is an indifferent fixed point then $f(D_r(b_{i,j}))=D_{t_{i,j}}(b_{i,j})$ with $t_{i,j}=r$ and $f:D_r(b_{i,j})\to D_{t_{i,j}}(b_{i,j})$ is bijective.
    \end{itemize}
    Let $\delta_1,\ldots,\delta_n$ as in proof of Theorem \ref{T1}. Let us choose $\varepsilon>0$ such that $$\varepsilon<\min_{1\leq i\leq n}\{\delta_i,t_{i,1}\ldots,t_{i,N_i}\},$$and let us define $F_\varepsilon$ as in \eqref{F} in proof of Theorem \ref{T1}.\\
    By Theorem \ref{T1}, $|F_\varepsilon(x)-f_i(x)|<\varepsilon$ for all $x\in B_{r_i}(a_i)$, and all $i=1,\ldots,n$. In particular, by (3) in Remark \ref{Rem}, $F_\varepsilon(D_r(b_{i,j}))=D_{t_{i,j}}(b_{i,j})$ for all $i=1,\ldots,n$ and all $j=1,\ldots,N_i$.\\
    Following the ideas from the proofs of Corollaries \ref{C1}, \ref{C2} and \ref{C3}, we have that 
    \begin{itemize}
        \item if $b_{i,j}$ is an attracting fixed point of $f_i$ then $F_\varepsilon(D_r(b_{i,j}))\subsetneq D_r(b_{i,j})$, and hence by (a) in Proposition \ref{T418}, $F_\varepsilon$ has a unique attracting fixed point in $D_r(b_{i,j})$.
        \item if $b_{i,j}$ is a repelling fixed point of $f_i$ then $F_\varepsilon(D_r(b_{i,j}))\supsetneq D_r(b_{i,j})$ and $F_\varepsilon:D_r(b_{i,j})\to D_{t_{i,j}}(b_{i,j})$ is bijective. Therefore, by (c) in Proposition \ref{T418}, $F_\varepsilon$ has a unique repelling fixed point in $D_r(b_{i,j})$.
        \item If $b_{i,j}$ is an indifferent fixed point of $f_i$ then $F_\varepsilon(D_r(b_{i,j}))=D_r(b_{i,j})$ and $F_\varepsilon:D_r(b_{i,j})\to D_{t_{i,j}}(b_{i,j})$ is bijective. Moreover, if the hypothesis of Corollary \ref{C3} hold for $b_{i,j}$, by Corollary \ref{C3}, $F_\varepsilon$ has a indifferent fixed point in $B_{r_i}(a_i)$. 
    \end{itemize}
    Therefore, $F_\varepsilon$ contains $n_i$ attracting, $m_i$ repelling and $l_i$ indifferent fixed point in $B_{r_i}(a_i)$ for each $i=1,\ldots,n$. 
\end{proof}

\section{Examples}
\label{sec:examples}
This section presents examples to elucidate the assertions made in Theorem~\ref{T1}.

\begin{Example}
In $\mathbb{C}_3$. Let us consider the rational functions 
\begin{align*}
    f_1:B_{r_1}(0)&\to B_{t_1}(0)&\hspace{1cm}\text{and}\hspace{1cm}f_2:B_{r_2}(3)&\to B_{t_2}(3)\\
     z&\mapsto\alpha z & z&\mapsto\beta z(z-3)+z
\end{align*}
with $r_1,r_2<1/3$. Note that $0$ is a fixed point of $f_1$ with multiplier $\alpha$, and $3$ is a fixed point of $f_2$ with multiplier $3\beta+1$.

Following the proof of Theorem \ref{T1}, note that $\delta_1=\delta_2=\frac{1}{3}$, $s_i=\sqrt{\frac{r_i}{3}}$, for $i=1,2$. Then we choose $c_1,c_2\in\C$ with $|c_i|=s_i$, for $i=1,2$.

Therefore, the rational functions $h_1$ and $h_2$ are defined as $$h_1(z)=\frac{1}{1-\left(\frac{z}{c_1}\right)^{M_1}}\hspace{1cm}\text{and}\hspace{1cm}h_2(z)=\frac{1}{1-\left(\frac{z-3}{c_2}\right)^{M_2}},$$for some $M_1,M_2\geq1$, and then their na-surgery is: 
$$F_\varepsilon(z)=f_1(z)h_1(z)+f_2(z)h_2(z)=\frac{\alpha z}{1-\left(\frac{z}{c_1}\right)^{M_1}}+\frac{\beta z(z-3)+z}{1-\left(\frac{z-3}{c_2}\right)^{M_2}}.$$

Note that $0$ is a fixed point of $F_\varepsilon$, moreover, since
\begin{itemize}
    \item $f_1(0)=f_2(0)=0$,
    \item $f_1'(0)=\alpha$, $f_2'(0)=1-3\beta$,
    \item $h_1(0)=1$, $h_2(0)=\tfrac{1}{1-\left(\tfrac{-3}{c_2}\right)^{M_2}}$
    \item and since $|c_2|<\tfrac{1}{3}$, then
    $\left|\tfrac{-3}{c_2}\right|>1$
\end{itemize}
then: 
\begin{align*}
F_\varepsilon'(0)&=f_1'(0)h_1(0)+f_1(0)h_1'(0)+f_2'(0)h_2(0)+f_2(0)h_2'(0)\\
        &=f_1'(0)+f_2'(0)h_2(0)\\
        &=\alpha+\tfrac{1-3\beta}{1-\left(\tfrac{-3}{c_2}\right)^{M_2}}.
\end{align*}
If $1-3\beta=0$, then $F_\varepsilon(0)=\alpha$. If $1-3\beta\neq0$, we choose $0<\varepsilon<\frac{1}{|1-3\beta|}$, then  $\left(\frac{r_2}{\delta_2}\right)^{M_2/2}<\ \varepsilon$, and hence:
$$\left|\frac{1-3\beta}{1-\left(\frac{-3}{c_2}\right)^{M_2}}\right|=\frac{|1-3\beta|}{\left|\frac{-3}{c_2}\right|^{M_2}}=|1-3\beta|\left(\frac{r_2}{\delta_2}\right)^{M_2/2}<1.$$
Therefore, if $0$ is an attracting fixed point of $f_1$ then $|\alpha|<1$, which implies that $|F_\varepsilon'(0)|<1$ and hence $0$ is an attracting fixed point of $F_\varepsilon$. Analogously, if $0$ is an indifferent (resp. repelling) fixed point of $f_1$, then $0$ is an indifferent (resp. repelling) fixed point of $F_\varepsilon$.  
\end{Example}

\begin{Example}
In $\C_3$. Let us assume that we do not know the rational functions $f_1,\ldots,f_n$; instead, we only know some about the dynamics over the balls. For example let us assume that certain dynamics implies that $$B_{\frac{1}{3^2}}(0)\to B_{\frac{1}{3^3}}(0),\quad B_{\frac{1}{3^2}}(3)\to B_{\frac{1}{3}}(3)\quad\text{ and }\quad B_{\frac{1}{3^2}}(6)\to B_{\frac{1}{3^2}}(6).$$
Then, our goal is to find a rational function, $F\in\C_3(z)$, such that $$F(B_{\frac{1}{3^2}}(0))=B_{\frac{1}{3^3}}(0),\quad F(B_{\frac{1}{3^2}}(3))=B_{\frac{1}{3}}(3),\quad F(B_{\frac{1}{3^2}}(6))=B_{\frac{1}{3^2}}(6),$$this rational function will be an approximation of the original dynamics.\\
Since we do not know the functions $f_1,f_2$, and $f_3$, we can use simple rational functions that respect the dynamics. For example, $$f_1(z)=3z,\quad f_2(z)=\frac{z+6}{3}\quad\text{ and }\quad f_3(z)=z.$$\\
Note that
\begin{itemize}
    \item 
    \begin{itemize}
        \item $f_1(B_{\frac{1}{3^2}}(0))=B_{\frac{1}{3^3}}(0)$,
        \item $f_2(B_{\frac{1}{3^2}}(3))=B_{\frac{1}{3}}(3)$ and
        \item $f_3(B_{\frac{1}{3^2}}(6))=B_{\frac{1}{3^2}}(6)$
    \end{itemize}
    \item Following the proof of Theorem \ref{T1}, $r_1=r_2=r_3=\frac{1}{3^2}$, $\delta_1=\delta_2=\delta_3=\frac{1}{3}$, and hence $$s_i=\sqrt{r_i\delta_i}=\frac{1}{3\sqrt{3}},$$for $i=1,2,3$. Therefore, we can choose $c=c_1=c_2=c_3\in\C_3$ such that $|c|=\frac{1}{3\sqrt{3}}$. 
    \item Following the proof of Theorem \ref{T1}, $t_1=\frac{1}{3^3}$, $t_2=\frac{1}{3}$ and $t_3=\frac{1}{3^2}$, therefore, for $\varepsilon\geq\frac{1}{3^3}$ we have that $\tau=\frac{1}{3^3}$.
    \item Since $\frac{r_i}{\delta_i}=\frac{1/3^2}{1/3}=\frac{1}{3}$, then $$\left(\frac{r_i}{\delta_i}\right)^{M_i/2}<\tau\iff\frac{1}{3^{M_i/2}}<\frac{1}{3^3}\iff M_i>6,$$ therefore, we can choose $M_i=7$ for $i=1,2,3$.
\end{itemize}
Hence, let us define $$h_1(z)=\frac{1}{1-\left(\frac{z}{c}\right)^7},\quad h_2(z)=\frac{1}{1-\left(\frac{z-3}{c}\right)^7},\quad h_3(z)=\frac{1}{1-\left(\frac{z-6}{c}\right)^7},$$and $$F(z)=\frac{3z}{1-\left(\frac{z-3}{c}\right)^7}+\frac{z+6}{3\left(1-\left(\frac{z-6}{c}\right)^7\right)}+\frac{z}{1-\left(\frac{z-3}{c}\right)^7}.$$
By Theorem \ref{T1} the rational function $F$ satisfies $$F(B_{\frac{1}{3^2}}(0))=B_{\frac{1}{3^3}}(0),\quad F(B_{\frac{1}{3^2}}(3))=B_{\frac{1}{3}}(3),\quad F(B_{\frac{1}{3^2}}(6))=B_{\frac{1}{3^2}}(6).$$
\end{Example}

In this example, we see that our method solves, up to an approximation, the inverse problem given a global rational function that describes the desired dynamics in the given balls.

%Cambiar 
\bibliographystyle{alpha}
\renewcommand{\bibname}{References} 
 \bibliography{references} 

\begin{thebibliography}{Gam03}

\bibitem[Ben02]{benedetto2002components}
Robert~L Benedetto.
\newblock Components and periodic points in non-archimedean dynamics.
\newblock {\em Proceedings of the London Mathematical Society}, 84(1):231--256,
  2002.

\bibitem[Ben19]{benedetto2019dynamics}
Robert~L Benedetto.
\newblock {\em Dynamics in one non-archimedean variable}, volume 198.
\newblock American Mathematical Soc., 2019.

\bibitem[Gam03]{gamelin2003complex}
Theodore Gamelin.
\newblock {\em Complex analysis}.
\newblock Springer Science \& Business Media, 2003.

\bibitem[Hsi96]{hsia1996weak}
Liang-Chung Hsia.
\newblock A weak n{\'e}ron model with applications to $ p $-adic dynamical
  systems.
\newblock {\em Compositio Mathematica}, 100(3):277--304, 1996.

\bibitem[Kiw14]{MR3265299}
Jan Kiwi.
\newblock Puiseux series dynamics of quadratic rational maps.
\newblock {\em Israel J. Math.}, 201(2):631--700, 2014.

\bibitem[NC23]{MR4510115}
V\'{\i}ctor Nopal-Coello.
\newblock Non-{A}rchimedean indifferent components of rational functions that
  are not disks.
\newblock {\em Trans. Amer. Math. Soc.}, 376(1):419--451, 2023.

\bibitem[RL03]{rivera2003dynamique}
Juan Rivera-Letelier.
\newblock Dynamique des fonctions rationnelles sur des corps locaux.
\newblock {\em Ast{\'e}risque}, 287:147--230, 2003.

\bibitem[Shi87]{shishikura1987quasiconformal}
Mitsuhiro Shishikura.
\newblock On the quasiconformal surgery of rational functions.
\newblock In {\em Annales scientifiques de l'{\'E}cole Normale Sup{\'e}rieure},
  volume~20, pages 1--29, 1987.

\bibitem[Sil07]{silverman2007arithmetic}
Joseph~H Silverman.
\newblock {\em The arithmetic of dynamical systems}, volume 241.
\newblock Springer Science \& Business Media, 2007.

\end{thebibliography}
\end{document}